\documentclass[12pt,a4paper]{amsart}
\usepackage{graphics}
\usepackage{epsfig}
\usepackage{graphicx}
\theoremstyle{plain}
\usepackage{amssymb}
\usepackage[ukrainian,english]{babel}
\advance\hoffset-20mm \advance\textwidth40mm

\newtheorem{theorem}{Theorem}
\newtheorem{lemma}{Lemma}
\newtheorem*{theo*}{Theorem}

\newtheorem{corollary}{Corollary}
\theoremstyle{definition}

\newtheorem*{definition*}{Definition}

\DeclareMathOperator{\Ker}{Ker}

\DeclareMathOperator{\Aut}{Aut}

\DeclareMathOperator{\End}{End}
\DeclareMathOperator{\Soc}{Soc}

\def\K{\mathbb{K}}

\begin{document}
\sloppy
\title[Nilpotent modules over polynomial rings]
{Nilpotent modules over polynomial rings}
\author
{Ie.Yu.Chapovskyi, A. P.  Petravchuk}
\address{Ie.Yu.Chapovskyi: Department of Algebra and Mathematical Logic, Faculty of Mechanics and Mathematics,
Taras Shevchenko National University of Kyiv, 64, Volodymyrska street, 01033  Kyiv, Ukraine}
\email{safemacc@gmail.com}
\address{Anatoliy P. Petravchuk:
Department of Algebra and Mathematical Logic, Faculty of Mechanics and Mathematics,
Taras Shevchenko National University of Kyiv, 64, Volodymyrska street, 01033  Kyiv, Ukraine}
\email{aptr@univ.kiev.ua , apetrav@gmail.com}
\date{\today}

%
\begin{abstract}
Let $\mathbb K$ be an algebraically closed field of characteristic zero,
$\mathbb K[X]$  the polynomial ring in $n$ variables.
The vector space $T_n = \mathbb K[X]$ is a  $\mathbb K[X]$-module with the action $x_i \cdot v = v_{x_i}'$ for $v \in T_n$. Every  finite dimensional submodule $V$  of $T_n$ is  nilpotent, i.e. every polynomial 
$f \in \mathbb K[X]$ with zero constant term acts nilpotently (by multiplication) on  $V.$  We prove that every  nilpotent $\mathbb K[X]$-module $V$ of finite dimension over $\mathbb K$ with one dimensional socle can be isomorphically embedded in the module $T_n$. The  automorphism groups of the module $T_n$ and its finite dimensional  monomial submodules are found. Similar results are obtained for (non-nilpotent) finite dimensional $\mathbb K[X]$-modules with one dimensional socle.
 \end{abstract}
\maketitle

\section{Introduction}
Let $\K$ be an arbitrary algebraically closed field of characteristic zero, $\K[X]:=\K[x_1, \ldots , x_n]$ the polynomial algebra  in $n$ variables. Every  finite-dimensional module $V$ over $\K[X]$
determines $n$ pairwise commuting matrices $A_1,\dots,A_n$ defining the action of the elements $x_1,\dots,x_n$ of the algebra $\K[X]$ on $V$ in some fixed basis of $V$. Classification (up  to similarity) of such tuples ($A_1,\dots,A_n$) of square matrices   is a wild problem in case of $n \ge 2$ (see \cite{GP}), and therefore classifying finite-dimensional $\K[X]$-modules  is a wild problem as well. We consider  $\K[X]$-modules $V$ with a  restriction: $V$ has only one   minimal submodule, i.e. $\dim_{\K}\Soc(V)=1.$  A natural example of such a module is the linear space $T_n:=\K[X]$ over the field $\K$ with the following action of the algebra $\K[X]$ on $T_n$: $x_i f = \frac{\partial}{\partial x_i}(f)$ for all $f \in T_n$ (this module is infinite dimensional, and   its every nonzero finite dimensional submodule has one dimensional socle).
Finite dimensional submodules of the module $T_2$ were studied in \cite{Gonzalez} (such modules  appeared there as polynomial translation modules). Some  properties of modules over polynomial rings were studied in  \cite{Qui,Sus}. Note that endomorphisms of modules over the polynomial ring in one variable were studied in \cite{Best}.

Each finite-dimensional module $V$ over the algebra $\K[X]$ is also  a module over the abelian Lie algebra $\K\langle S_1, \ldots , S_n\rangle$  spanned by the linear operators $S_i$ which define  the action of the elements $x_i$ on the module $V$ (about connection between Lie algebras and modules over  polynomial rings see \cite{PS}).
Let $V$ be a finite dimensional module over the Lie algebra $\K\langle S_1, \ldots , S_n\rangle .$  Then $V$ is a direct sum of root submodules $V = V_{\lambda_1}\oplus \dots \oplus V_{\lambda_k}$, where $\lambda_i$ are roots on the Lie algebra $\K\langle S_1, \ldots , S_n\rangle .$
 Each element $a$ from $\K\langle S_1, \ldots , S_n\rangle$ induces on $V_{\lambda_i}$ a linear operator $\overline{a_{i}}$ with the only  eigenvalue $\lambda_i(a)$.
 Hence, it is sufficient to consider root subspaces only and to suppose that the elements of $\K\langle S_1, \ldots , S_n\rangle$ induce linear operators  on the module $V$ with one  eigenvalue. First, we consider modules, where this eigenvalue $\lambda(a)$ is equal to zero. We  call such modules (over polynomial rings) nilpotent. In fact, for such a module $V$  there is  a number $k$ with $(\K_{0}[X])^k V = 0,$ where $\K_{0}[X]$ is the ideal of $\K[X]$ consisting of polynomials with zero constant term.

The main results of the paper: it is shown that every nilpotent finite dimensional $\K[X]$-module with one dimensional socle is isomorphic to a submodule of the module $T_n$ (Theorem \ref{nil}).  This result can also be  extended on non-nilpotent modules with one dimensional socle. It is proved that every finite dimensional module with one dimensional socle can be isomorphically embedded in the module $D_{\lambda {_1}, \ldots ,\lambda _{n}}$ for some $\lambda {_1}, \ldots ,\lambda _{n}\in \K$ (Corollary \ref{nonnilpoten}). Automorphism groups of the $\K[X]$-module $T_n$ and its finite dimensional submodules are studied. In particular, it is proved that $\Aut(T_n) \simeq \K[[X]]^*,$ where $\K[[X]]^*$ is the multiplicative group of the algebra $\K[[X]]$ of formal power series  in $n$ variables  (Theorem \ref{automorphisms}).
As a  consequence, the $\K[X]$-module $T_n$ contains  exactly one isomorphic copy of each finite dimensional nilpotent $\K[X]$-module with one-dimensional socle (Corollary \ref{unique}).

We use  standard notation. The ground  field $\K$ is of characteristic zero,  $\K[[X]]:=\K[[x_1, \ldots , x_n]]$  is the algebra of formal power series  in $n$ variables. We denote by  $\K_{0}[X]$ the ideal of the polynomial algebra $\K[X]$ consisting of polynomials with zero constant term.
Recall that the socle $\Soc(M)$ of a module $M$ is the  sum of all its minimal submodules. A $\K[X]$-module $V$ is called nilpotent if $(\K_{0}[X])^{k}V=0$ for some natural $k.$ A $\K[X]$-module $V$ is said to be locally nilpotent if for any $ a \in \K_{0}[X],  v \in V\ $ there exists a number $ k = k(a,v) $ such that $ a^k v = 0.$
If $V$ is a nilpotent finite dimensional $\K[X]$-module, then $\Soc(V)=\bigcap_{i=1}^n \ker S_i,$ where $S_i$ is the linear operator on $V$ induced by the element $x_i \in \K[X].$
Further, $\End(V)$  denotes the $\K$-algebra of all endomorphisms of the module $V.$  If $V$ is a  $\K[X]$-module  and $\theta$ is an automorphism of the algebra $\K[X],$ then  $V_{\theta}$ is the "twisted" \ $\K[X]$-module with the action: $f\circ v=\theta (f)\cdot v, \ f\in\K[X], v\in V.$

\section{The universal $\K[X]$-module with one dimensional socle}
Recall that  $T_n$  denotes the $\K[X]$-module $\K[X]$ with the action of the algebra generators $x_i$ on $T_n$ by the rule:
$x_i f = \frac{\partial}{\partial x_i}(f), i=1, \ldots ,n, \; f \in T_n$. It is easy to see that
$T_n$ is a locally nilpotent $\K[X]$-module and $\Soc(T_n) = \K[1]$ is a one-dimensional submodule. The following statement is well known (and can be derived from standard facts about de Rham cohomologies) but having no  exact references we supply  it with a  proof.
 \begin{lemma}\label{potential}
Let $f_1,\dots,f_k, \; k \le n$  be elements of $\K[x_1,\dots,x_n].$  If the polynomials $f_1,\dots,f_k$ satisfy the conditions
		$\frac{\partial f_i}{\partial x_j} = \frac{\partial f_j}{\partial x_i},\; i,j = 1,\dots,k,$ then there exists a polynomial $h \in \K[x_1,\dots,x_n]$  such that $\frac{\partial h}{\partial x_i} = f_i,\; i = 1,\dots,k$.
		The polynomial $h$ is determined by polynomials $f_1,\dots,f_k$ uniquely up to a summand from $\K[x_{k+1},\dots,x_n]$.
\end{lemma}
\begin{proof}
Uniqueness is obvious. Let us build such a polynomial $h$  by induction on $k$.
		If  $k=1,$ then the polynomial  $h(x) = \int_0^{x_1} f_1 dt$ satisfies the  conditions of lemma (here the variable $x_1$ of the polynomial $f_1$ is denoted by $t$). 		
		By inductive hypothesis,  a polynomial $h_1 \in \K[x_1,\dots,x_n]$ is built  such that
		$$\frac{\partial h_1}{\partial x_i} = f_i,\quad i = 1,\dots,k-1.$$
		Setting $h = h_1 + \int_0^{x_k} (f_k -
		\frac{\partial h_1}{\partial x_k}) dx_k$ we have for all $i < k$
		$$ \frac{\partial h}{\partial x_i}  =
		\frac{\partial h_1}{\partial x_i} + \frac{\partial}{\partial x_i}
		\int_0^{x_k} \left( f_k - \frac{\partial h_1}{\partial x_k}\right) dx_k = f_i + \int_0^{x_k} \left( \frac{\partial f_k}{\partial x_i} -
		\frac{\partial f_i}{\partial x_k } \right) dx_k = f_i .$$
		It is easy to show that  $\frac{\partial h}{\partial x_k} = f_k$, so the polynomial  $h$ satisfies  the conditions of the lemma.
\end{proof}

The following statement shows that $T_n$ is  universal (in some sense) for  finite dimensional nilpotent $\K[X]$-modules with one dimensional socle.

\begin{theorem} \label{nil}
	Every  finite dimensional nilpotent $\K[X]$-module $V$ with one dimensional socle  can be isomorphically embedded in the $\K[X]$-module $T_n$.
\end{theorem}
\begin{proof}
	Induction on $k = \dim_{\K} V$. If $k=1$, then
	$V = K\langle v\rangle $ for any non-zero element $v \in V$.
	We define the embedding $\varphi: V \to T_n$ by the rule $\varphi(v) = 1$ and further  by linearity.
	Let $\dim_{\K} V = k > 1$ and $W$  be a submodule in $V$ of codimension $1$ in $V$ (such  a submodule does exist because the module $V$ is nilpotent).
	Take any element $v_0 \in V \setminus W$.
	Then  $x_i v_{0} \ne 0$ for some $i, 1 \le i \le n$.
	Indeed, otherwise $\K\left< v_0 \right>$ is a one-dimensional submodule, and therefore
	$\Soc(V) = \K\left< v_0 \right>.$ The latter  is impossible  since $\Soc(V)$ is obviously contained in $W$. The contradiction shows that $x_i v_{0} \ne 0$ for some $i, 1 \le i \le n.$  By the inductive hypothesis there exists  an isomorphism $\varphi: W \to M$, where $M$ is a submodule in $T_n$. Denote by $S_i$ the linear operator on $V$  induced by the element $x_i$, i.e.
	$S_i(v) = x_i v, \; v \in V,\, i = 1,\dots, n$.
	Then  $S_1,\dots,S_n$ are pairwise commuting nilpotent operators on $V$ and $\bigcap_{i=1}^n \ker S_i$ is a one dimensional submodule which coincides with the socle $\Soc(V)$.
	Since $\dim_{\K} V/W = 1,$ we have 	$S_i(v_0) \in W,\; i = 1\dots,n$. Therefore, $\varphi(S_i(v_0)) \in M,\; i = 1,\dots n$ and $\varphi(S_i(v_0))$ are polynomials in $x_1,\dots x_n$.
	Let us denote $f_i = \varphi(S_i(v_0))$. The map  $\varphi$ is an isomorphism of
	$\K[X]$-modules,  so the following equalities hold
	\begin{align*}
	\varphi(S_i(S_j(v_0))) = \frac{\partial}{\partial x_i}
	\left( \varphi(S_j(v_0)) \right) =
	\frac{\partial f_j}{\partial x_i} \\
	\varphi(S_j(S_i(v_0))) = \frac{\partial}{\partial x_j}
	\left( \varphi(S_i(v_0)) \right) =
	\frac{\partial f_i}{\partial x_j}
	\end{align*}
	Since $S_i S_j = S_j S_i$, we have
	\begin{align}	\label{potential}
	\frac{\partial f_i}{\partial x_j} =
	\frac{\partial f_j}{\partial x_i},\quad i,j = 1,\dots,n.
	\end{align}
	By Lemma \ref{potential} there exists  a polynomial  $h\in \K[X]$  such that
$\frac{\partial h}{\partial x_i}=f_i,  i=1, \ldots , n.$
	It is easy to see that  $M + \K \left < h \right >$ is a submodule of $T_n$.
	Let us show that $h \notin M$. Indeed, suppose
	$h \in M$. Set $v_1 = \varphi^{-1}(h) \in W$
	and note that $\varphi^{-1}$ is an
	 isomorphism of the $\K[X]$-modules $M$ and $W$.
	Therefore the following equalities hold:
$$
	\varphi^{-1}
	\left(
	\frac{\partial h}{\partial x_i}
	\right) =
	S_i(\varphi^{-1}(h)),\quad i = 1,\dots,n.
$$
	On the other hand, we have
$$
 	\varphi^{-1}
	\left(
	\frac{\partial h}{\partial x_i}
	\right) =
	\varphi^{-1}(f_i) = S_i(v_0),
$$
	and therefore  $ 	S_i(v_1 - v_0) = 0, \quad i = 1,\dots,n.$
	The last equalities mean  that $v_1 - v_0 \in \Soc(V) =
	\K \left < v_0 \right >$, thus $v_1 - v_0 \in W$.
	Since $v_1 \in W$ by above mentioned, we obtain $v_0 \in W.$ The latter contradicts to the choice of  $v_0$. The contradiction shows that $h \notin M$.
	Define a $\K$-linear map
$$
	\psi:V = W + \K \left < v_0 \right > \to M +
	\K \left < h \right >
$$
	by the rule: $ 	\psi(w) = \varphi(w),   \  w \in W, \ 	\psi(v_0) = h,$
	and further by linearity.
	
	Straightforward check shows that
	$\psi$ is an  isomorphism of the  $\K[X]$-modules $V$ and
	 $M + \K \langle  h\rangle \subset T_n.$ The proof is complete.
\end{proof}
 Theorem \ref{nil} is easily extended on non-nilpotent modules. Consider the linear space $P_{n}:=\K[[X]] $ of the algebra of formal power series in $n$ variables,   and define on it the action of the algebra $\K[X]$ in  the same way as for the module $T_n.$ Let $\lambda_1,\dots,\lambda_n$ be arbitrary elements of the field $\K.$ The product   $D_{\lambda {_1}, \ldots ,\lambda _{n}}:=\exp(\lambda_1 x_1 + \dots + \lambda_n x_n)T_n$ is obviously a submodule of the $\K[X]$-module $P_n.$    It is easy to see  that $\lambda_1,\dots,\lambda_n$ are eigenvalues of the operators  induced by $x_1,\dots,x_n$ respectively on $D_{\lambda {_1}, \ldots ,\lambda _{n}}.$
\begin{corollary}\label{nonnilpoten}
	Let $V$ be a finite-dimensional $\K[X]$-module with one dimensional socle. Then $V$ is isomorphic to a submodule of  $D_{\alpha_1,\dots,\alpha_n}$ for some $\alpha_1,\dots,\alpha_n \in \K$.
\end{corollary}
\begin{proof}
	Write $\Soc V$ as $\K \left < v_0 \right >$ for some $v_0\in V, v_0\not =0.$ Then  $x_i v_0 = \alpha_i v_0$ for some $\alpha_i \in \K,\; i=1,\dots,n$. Let  $\theta$ be the automorphism of the polynomial algebra
	$\K[X]$  defined by the rule:  $x_i \mapsto x_i - \alpha_i, \ i=1, \ldots , n.$ Denote by  $V_{\theta}$ the corresponding twisted $\K[X]$-module. It is easy to see that the module $V_{\theta}$ is nilpotent (because the elements $x_1,\dots,x_n\in \K[X]$ induce nilpotent linear operators on $V_{\theta}$).
	By Theorem \ref{nil}, the  $\K[X]$-module $V_{\theta}$ is isomorphic to a submodule $W$
	of the  $\K[X]$-module $T_n$. The product
	$\exp(\alpha_1 x_1 + \dots + \alpha_n x_n) W$ is obviously a submodule of the $\K[X]$-module  $D_{\alpha_{1}, \ldots , \alpha _{n}}.$  Straightforward check shows that  $\exp(\alpha_1 x_1 + \dots + \alpha_n x_n) W \simeq W_{\sigma}$, where $\sigma$ is the automorphism of the polynomial algebra $\K[X]$ defined by $\sigma(x_i)=x_i+\alpha_i,\; i = 1,\dots,n.$
 Thus we have $W_{\sigma} \simeq (V_{\theta})_{\sigma} \simeq V_{\theta \sigma} \simeq V$.
	But then  $V \simeq W_{\sigma},$ and $W_{\sigma}$ is a submodule of the module
	$\exp(\alpha_1 x_1 + \dots + \alpha_n x_n) T_n = D_{\alpha_1,\dots,\alpha_n}$.
\end{proof}

\section{Automorphism groups of modules with one dimensional socle }

\begin{lemma}\label{decrease}
	Let $\varphi \in \End(T_n)$.	
	Then $\deg_{x_{i}} \varphi(f) \le \deg_{x_{i}} f$ for every polynomial $f\in T_n, \ \  i=1,\dots,n.$
\end{lemma}
\begin{proof}
	Since $\varphi$ is an endomorphism of the module $T_n$  we have for every
	$f \in T_n$
	$$
	\varphi \left(\frac{\partial^k}{\partial x_i^k}f \right) =
	\frac{\partial^k}{\partial x_i^k} \varphi(f), \quad i = 1,\dots,n.
	$$
	Then  the following relations yield the result:
	$$
	\deg_{x_{i}} \varphi(f) =
	\max\{k \ge 0: \frac{\partial^k}{\partial x_i^k} \varphi(f) \ne 0\} = 	\max\{k \ge 0: \varphi(\frac{\partial^k}{\partial x_i^k} f) \ne 0\} \le $$
	$$\le \max\{k \ge 0: \frac{\partial^k}{\partial x_i^k} f \ne 0\} =
	\deg_{x_{i}}f, \quad i = 1,\dots,n.
	$$
\end{proof}
We   further   use the following multi-indices
$$
	\alpha := (\alpha_1,\dots,\alpha_n),  \  \ \
	\alpha! := \prod_{i=1}^n \alpha_i!,  \  \ \
	x^{\alpha} := \prod_{i=1}^n x_i^{\alpha_i}, \ \ \
	\frac{\partial^{\alpha}}{\partial x^{\alpha}} := \frac{\partial^{\alpha_1+\dots+\alpha_n}}
	{\partial x_{1}^{\alpha_1}\dots\partial x_{n}^{\alpha_n}}.
$$

\begin{theorem}\label{endo}
	A linear transformation $\varphi$ of the vector space $\K[X]$ is an
	endomorphism of the $\K[X]$-module $T_n$ if and only if it can be written in the form
	$$
	\varphi =
	\sum_{\substack{(\alpha_1,\dots,\alpha_n)\\
			\alpha_i \ge 0}}
	c_{\alpha_1 \,\dots,\alpha_n} \frac{\partial^{\alpha_1+\dots+\alpha_n}}
	{\partial x_{1}^{\alpha_1}\dots\partial x_{n}^{\alpha_n}}, \ \ c_{\alpha_1 \,\dots,\alpha_n}\in\K .
$$
	The coefficients here are the form
$$
	c_{\alpha_1 \,\dots,\alpha_n} =
	\varphi \left(\frac{x_1^{\alpha_1}}{\alpha_1!} \dots 					  	 \frac{x_n^{\alpha_n}}{\alpha_n!} \right)(0).
$$
\end{theorem}
\begin{proof}
 Necessity.		
	Let  the linear transformation $\varphi$ be an endomorphism of the $\K[X]$-module $T_n$. 	
	Consider the map  $N: T_n\to T_n$ defined by the rule:
	$$
	N(f) = \varphi(f) - \sum_{\alpha:\,\alpha_i \ge 0} \varphi \left(\frac{x^{\alpha}}{\alpha!}\right)(0)\,
	\frac{\partial^{\alpha}}{\partial x^{\alpha}} f, \quad
	f \in T_n
	$$
	($N(f)$ is  correctly defined because the linear operators $\frac{\partial^{\alpha}}{\partial x^{\alpha}}$ on $T_n$ are locally nilpotent).  Let us  show (by induction on the total degree of the polynomial $f$) that  the map $N$  is identically zero.  Suppose  that $\deg f = 0$.
	Then we have 	$N(f) = \varphi(f) - \varphi(f)(0) = 0,$	since by Lemma \ref{decrease}  the endomorphism $\varphi$ does not increase $\deg f.$ 	Let $\deg f > 0$.
	The map $N$ is an endomorphism of the $\K[X]$-module $T_n$ (as a linear combination of  endomorphisms).
	By induction hypothesis
$$
	 N(\frac{\partial}{\partial x_i}f) = 0=\frac{\partial}{\partial x_i} N(f) ,
	\quad i = 1,\dots n,
$$
	 so we obtain
$$
	N(f) = N(f)(0) = \varphi(f)(0) - \sum_{\alpha:\,\alpha_i \ge 0} \varphi \left(\frac{x^{\alpha}}{\alpha!}\right)(0)\,
	\frac{\partial^{\alpha} f}{\partial x^{\alpha}}(0) = $$
$$=\varphi(f)(0) - \sum_{\alpha:\,\alpha_i \ge 0} \varphi \left(\frac{\partial^{\alpha} f}{\partial x^{\alpha}}(0)\frac{x^{\alpha}}{\alpha!}\right)(0) =
	\varphi
	\left( f - \sum_{\alpha:\,\alpha_i \ge 0}
	\frac{\partial^{\alpha} f}{\partial x^{\alpha}}(0)
	\frac{x^{\alpha}}{\alpha!}
	\right)(0).
$$
	It follows from the last equality and Taylor's formula for the polynomial $f$  that $N(f) = 0.$
	The sufficiency is obvious.
\end{proof}

\begin{corollary}
	A linear trasformation $\varphi$ of the linear space $\K[X]$ is an automorphism of the  $\K[X]$-module $T_n$ if and only if it can be written in the form
	$$
	\varphi =
	\sum_{\substack{(\alpha_1,\dots,\alpha_n)\\
			\alpha_i \ge 0}}
	c_{\alpha_1 \,\dots,\alpha_n} \frac{\partial^{\alpha_1+\dots+\alpha_n}}
	{\partial x_{1}^{\alpha_1}\dots\partial x_{n}^{\alpha_n}}  \ \ \mbox{with} \
	 c_{(0,\dots,0)} \ne 0.$$
\end{corollary}

\begin{corollary}
	 Every submodule of the $\K[X]$-module $T_n$ is invariant under automorphism group  $\Aut(T_n).$
\end{corollary}
\begin{lemma}\label{twist}
Let $V$ be a module over the polynomial algebra $\K[X]$ and $\theta\in \Aut(\K[X]).$  Then $\Aut(V)=\Aut(V_{\theta}),$ where $V_\theta$ denotes the corresponding twisted module.
\end{lemma}
\begin{proof}
Note  that the modules $V$ i $V_{\theta}$ have the same underlying vector space $V,$  and therefore  $\Aut(V)$ and $\Aut(V_{\theta})$ are subgroups of the  group $GL(V).$  Let $\varphi \in \Aut(V_{\theta})$ be an arbitrary automorphism. Then    $\varphi (a\circ v)=a\circ \varphi (v)$ for all  $a\in \K[X]$ and $v\in V_{\theta}.$ Taking into account the  action of $\K[X]$ on $V_{\theta}$ we get $\varphi (\theta (a)\cdot v)=\theta(a)\cdot \varphi (v),$ where the symbol $(\cdot )$ denotes  the  action of  $\K[X]$ on $V.$
 Denoting the element $\theta (a)$ by $b$  we get $\varphi (b\cdot v)=b\cdot \varphi (v)$ for all $b\in \K[X]$ and $v\in V_{\theta}$ (we used here that $\theta :\K[X]\to \K[X]$ is a bijection). The latter  means that $\varphi$ is automorphism of the module $V.$
Thus, $\Aut(V_{\theta})\subseteq \Aut(V).$ One can prove analogously that $\Aut(V)\subseteq \Aut(V_{\theta}).$
\end{proof}
   We write down monomials from $\K[X]$ in the form $x^{\lambda}:=x_{1}^{\lambda _{1}}\cdots x_{n}^{\lambda _{n}},$ where $\lambda =(\lambda_1,\cdots,\lambda_n)$ is a multi-index. The following statement is well known.
\begin{lemma}\label{expon}
Let $\K_0[[X]]\subset \K[[X]]$  be the additive group of formal power series with zero constant term   and   $\K_1[[X]]^*$  the multiplicative group of formal power series with constant term which  equals one.
Then  $\exp:\K_0[[X]] \to \K_1[[X]]^*$ is a group isomorphism, and therefore
		$\K[[X]]^* = \K^* \times \exp(\sum_{\lambda,\lambda \ne 0}
		\K x^{\lambda}) \simeq \K^* \times \prod_{\lambda,\lambda \ne 0} \K$.
\end{lemma}
\begin{lemma}\label{max}
	Let $M,N$ be isomorphic proper submodules of the module $T_n$ and  $\varphi:M \to N$ be an isomorphism.
	Then there exist submodules $M_1,N_1 \subset T_n $  with $M_1 \varsupsetneq M$, $N_1 \varsupsetneq N,$  and an isomorphism $\varphi_1:M_1 \to N_1$ such that
	$\varphi_1 \vert_M = \varphi$.
\end{lemma}
\begin{proof}
	Since $M \ne T_n$, there are monomials that do not belong to $M$.
	Choose one of such  monomials, say  $x_1^{k_1} \dots x_n^{k_n}\in T_{n}\setminus M,$ of minimal total degree.
	Let
	\begin{align*}
	g_i = \varphi \left(\frac{\partial}{\partial x_i}
	x_1^{k_1} \dots x_n^{k_n} \right), \quad i = 1,\dots,n.
	\end{align*}
	Since $\varphi$ is an isomorphism between submodules,  we get by the definition of $g_i$
	$$
	\frac{\partial g_i}{\partial x_j} = \frac{\partial g_j}{\partial x_i},
	\quad i,j = 1,\dots,n.
$$
	By Lemma \ref{potential},  there exists a polynomial $g \in \K[X]$, such that
$	\frac{\partial g}{\partial x_i} = g_i, \quad i = 1,\dots, n.
$
	Note that $g \notin N$. Indeed, suppose $g\in N.$ Denote $f=\varphi ^{-1}(g).$ Then, $f\in M$ and
$$ \varphi (\frac{\partial }{\partial x_i}(f-x_1^{k_1} \dots x_n^{k_n}))=\frac{\partial }{\partial x_i}(\varphi (f))-g_{i}=\frac{\partial }{\partial x_i}(g)-g_{i}=0, \ i=1, \ldots , n.$$
The latter means that $\frac{\partial }{\partial x_i}(f-x_1^{k_1} \dots x_n^{k_n})=0, i=1, \ldots ,n, $ and therefore  $f=x_1^{k_1} \dots x_n^{k_n}+c,$ where $c\in\K$ is a constant. Since each non-zero submodule from $T_n$ obviously contains the field $\K ,$ we have that $x_1^{k_1} \dots x_n^{k_n}=f-c\in M.$ We obtain a contradiction to the choice of $x_1^{k_1} \dots x_n^{k_n}.$ Hence $g\not \in N.$
Set
$$
	M_1 = M \oplus \K\langle x_1^{k_1} \dots x_n^{k_n}\rangle, \ \ \ 	N_1 = N \oplus \K\langle g\rangle .$$ It is easy to see that $M_1, N_1$ are submodules from $T_n.$ Define the linear map $\varphi _{1}:M_1\to N_1$ by the rule: 	 $\varphi_1(m)=\varphi (m), m\in M,  \varphi_1(x_1^{k_1} \dots x_n^{k_n})=g$ and further by linearity.
	Straightforward check shows that  $\varphi_1: M_1\to N_1$ is an isomorphism.
\end{proof}

\begin{theorem}\label{lift}
	Let $M \subset T_n$ be a proper submodule and $\varphi :M\to T_n$ a monomorphism of $\K[X]$-modules.
	Then there exists an automorphism $\overset{\sim}{\varphi} \in \Aut(T_n)$ such that
	$\overset{\sim}{\varphi} \vert_{M} = \varphi$.
\end{theorem}
\begin{proof}
	Let  $\Psi$ be the set of all pairs of the form $(V, \psi ),$  where $V\supseteq M$ is a submodule from $T_n$ and $\psi :V\to T_n$  a monomorphism such that $\psi |_{M}=\varphi .$ 	Define a partial order on $\Psi$ by the following rule:
$$
(V_1, 	\psi_1) \preccurlyeq (V_2, \psi_2) \leftrightarrow V_1\subseteq V_2 \ \   \mbox{and } \ \  \psi_{2}|_{V_{1}}=\psi_1.
 $$
	If  $Z \subseteq \Psi$ is a chain, then there is an upper bound  $(V^{\star} , \psi^{\star}) $ for $Z$, where
$$
	V^* = \bigcup_{(V, \psi)\in Z}V, \ \ \mbox{and} \ \
   \psi^{\star}:V^{\star}\to T_n
$$
is defined by the rule: $\psi^{\star}(v)=\psi (v)$ for $v\in V $ with $(V, \psi )\in Z.$ By Zorn's lemma there is a maximal element $(\widetilde{V}, \widetilde{\psi})$ in $\Psi$.
	Let us show that $\widetilde{V}=T_n$ and $\widetilde{\psi}\in \Aut(T_n).$ Supposing that
	$\widetilde{V} \ne T_n$ and using Lemma \ref{max}
	we get a contradiction to the maximality of  $(\widetilde{V}, \widetilde{\psi}).$  Thus $\widetilde{V}=T_n.$  Further, if $\overset{\sim}{\psi}$ is not epimorphic, then applying Lemma \ref{max} to $({\widetilde \psi})^{-1}:\widetilde{\psi}(T_n) \to T_n,$ we have a contradiction to the fact that $\overset{\sim}{\psi}$ is a monomorphism. Hence,  $\overset{\sim}{\psi}$ is an automorphism of the module $T_n$ and the restriction $\overset{\sim}{\psi}$ on the submodule $M$ concides with the monomorphism $\varphi .$
\end{proof}
\begin{corollary}\label{unique}
	If $V_1$ and $V_2$ are submodules of the module $T_n$ and $V_1\not =V_2,$  then the modules $V_1$ and $V_2$ are not isomorphic.
\end{corollary}
We  call a submodule $M$ of the module $T_n$  {\it monomial} if it has a basis over the field $\K$ consisting of monomials.
\begin{theorem}\label{automorphisms}
1. The algebra $\End(T_{n})$ of all  endomorphisms of the module $T_n$ is isomorphic to  algebra $\K[[X]]$ of  formal power series in $n$ variables. In particular,  the automorphism group $\Aut(T_{n})$ is isomorphic to the multiplicative group $\K[[X]]^{\star}$ of the algebra $\K[[X]].$

2. Automorphism group of a finite dimensional monomial submodule $M\subset T_n$ of dimension $m$ over $\K$ is isomorphic to the direct product
$\K^{\star}\times (\K^{+})^{m-1},$ where $\K^{+}$ is the  additive group of  $\K.$

3. Automorphism group of any finite-dimensional $\K[X]$-module $V$ with one dimensional socle is isomorphic to the quotient group of $\K^{\star}\times (\K^{+})^{m}$ for some $m\geq 0.$

\end{theorem}
\begin{proof}
The first statement of the theorem follows from Theorem \ref{endo}.
Let us prove the second statement. Denote by $\Aut _{1}(M)$  the group of all automorphisms of the module $M,$ which act trivially on the submodule $\K\langle 1\rangle \subseteq M.$  Let us show that $\Aut _{1}(M)\simeq (\K^{+})^{m-1},$ where $m=\dim _{\K}M.$
Let $M = \K \langle x^{\lambda^1},\dots,x^{\lambda^m} \rangle ,$ where $\{ \lambda^1, \ldots , \lambda^m\}$ is a set  of multi-indices numbering a basis $\{ x^{\lambda^1},\dots,x^{\lambda^m}\}$ of the monomial module $M$ (the multi-index $\lambda =(0, \ldots , 0)$ belongs obviously to this set).  It is easy to see that an automorphism
		$\varphi = \sum_{\lambda} c_{\lambda} \frac{\partial}{\partial x^\lambda}$ of the module $T_n$ acts identically on $M$ if and only if
		$$
		c_{\lambda} =
		\begin{cases}
		1, \quad \lambda = (0,\dots,0) \\
		0, \quad \lambda \in \{\lambda^1,\dots,\lambda^m\}\setminus\{(0,\dots,0)\}
		\end{cases}.
		$$
		Obviously, such an automorphism can be written in the form
		$$\varphi = \exp(\sum_{\lambda \ne \{\lambda^1, \ldots , \lambda^m\}}c_{\lambda} \frac{\partial}{\partial x^\lambda}),$$
and, conversely, automorphism of such a  form acts trivially on $M$.
These automorphisms form the kernel $\Ker \psi$  of the restriction $\psi :\Aut _{1}(T_n)\to \Aut_{1}(M)$ (here $\Aut _{1}(T_n)$ denotes the group of automorphisms of $T_n$, which act trivially on the submodule $\K\langle 1\rangle \subset T_n$). Taking into account Lemma  \ref{expon}, we see that the isomorphism $\ln (1-x):  \K_1[[X]]^*\to \K_0[[X]]$ maps the subgroup $\Ker \psi$ onto the $\K$-subspace $G\subseteq \K_0[[X]]$ with the basis $\{x^{\lambda}, \lambda \not \in \{\lambda^1, \ldots , \lambda^m\}\}.$
 It is obvious that $\K_0[X]/G\simeq (\K^{+})^{m-1},$ and therefore $\Aut_{1}(M)\simeq (\K^{+})^{m-1}.$ The last isomorphism implies  $\Aut(M)\simeq \K^{\star}\times (\K^{+})^{m-1}.$		
Let us prove the last statement. It is obvious that every finite dimensional submodule $W\subset T_n$ lies in some finite dimensional monomial submodule $\overline W$ of $T_n.$ Since, by Theorem \ref{lift}, $\Aut(W)$ is a quotient of the group $\Aut({\overline W}),$ we see (using Lemma \ref{twist}) that $\Aut(W)$  is isomorphic to the quotient group of $\K^{\star}\times (\K^{+})^{m}$ for some $m\geq 0.$
\end{proof}


%

\begin{thebibliography}{99}
\bibitem{Best} P.Best, M.Gualtieri, P.Hayden, Orbits of the centralizer of a linear operator, J. Lie Theory, (2012) v.4, 1039-1048.
 \bibitem{GP}
I.~M.~Gelfand, V.~A.~Ponomarev,  Remarks on the classification of a pair of commuting linear transformations in a finite dimensional space, Funkc. Anal. Prilozhen. 3 (1969), no.4, 81--82.

\bibitem{Gonzalez}
A. Gonzalez-Lopez, N. Kamran and P. J. Olver, Lie algebras of differential operators in two complex variables
 Amer. J. Math. (1992) v.114, 1163--1185.

\bibitem{GG} M.Goto, F.Grosshans, Semisimple Lie algebras, Marcel Dekker, Inc. New York and Basel, 1978.

 \bibitem{Qui}  D. Quillen, Projective modules over polynomial rings, Invent. Math. 36 (1976),
167--171.

\bibitem{PS} A.P. Petravchuk, K.Ya. Sysak, Lie algebras associated with modules over polynomial rings, Ukrainian Math. J. 69 (2018), 1433--1444.

\bibitem{Sus} A.~Suslin, Projective modules over polynomial rings are free, Translated in "Soviet Mathematics" 17(4) (1976), 1160--1164.

\end{thebibliography}
\end{document}